\documentclass[10pt,a4paper]{article}

\usepackage{amsmath,amssymb,amsthm}

\newtheorem{theorem}{Theorem}
\newtheorem{lemma}{Lemma}

\newcommand{\m}{\mathfrak m}

\newcommand{\rto}{\rightarrow}

\begin{document}

\title{Local structure of \'etale algebra}
\author{Thierry Coquand}
\date{}
\maketitle

\section*{Introduction}

The goal of this note is to provide a constructive version of the proof of local
structure of \'etale algebra, attributed to Chevalley in \cite[18.4.6 (ii)]{Gr},
in the setting of constructive mathematics \cite{LQ}. As in \cite{Gr,Raynaud}, the
main tool is Zariski Main Theorem, that was already proved constructively in \cite{ACL}.

\section{Statement of the local result}

Let $R$ be a residually discrete local ring with maximal ideal $\m$ and residual field $k = R/\m$.
Let $E$ be a (finitely presented) \'etale $R$-algebra.

\begin{theorem}
  There exists $s_1,\dots,s_m$ in $E$ that are comaximal mod. $\m$, i.e. their images generate the unit ideal in $E/\m E$,
  and each 
  $E[1/s_i]$ is a standard \'etale $R$-algebra.
\end{theorem}

 Note that $E/\m E$ is a (finitely presented) unramified $k$-algebra and so, following \cite{LQ,L},
 is a finite product of monogene separable extensions of $k$.

 Using constructive Zariski Main Theorem \cite{ACL} we get that there exists a finite $R$-subalgebra $B$ of $E$
 and elements $s_i \in B$, comaximal in $E$, such that $B_i = B[1/{s_i}] \cong E[1/s_i]$ for all $i$.

Some $B_i/\m B_i$ may be the trivial algebra. We discard them, which explains why in the end
result, we only get a family which is comaximal {\em residually}. 

Each $B_i/\m B_i$ is a corresponding product of monogene separable extensions.
By refining $s_i$, we can assume that each $B_i/\m B_i$ is a itself a monogene separable extension of $k$.

\medskip

\section{Some lemmas about $B/\m B$}

Consider now $A = B/\m B$ which is a finite $k$-algebra. We still write $s_i$ for its class in $A$.
Then $A[1/s_i]$ is an extension of $k$ of the form $k[X]/(p)$ for some separable polynomial $p$.

\begin{lemma}
  We have $u$ and $N$ such that $s_i^N(1-s_iu) = 0$. If we write $e_i = (s_iu)^N$, we have an idempotent $e_i$ of $A$ such that
  $A[1/e_i] = A[1/s_i]$.
\end{lemma}

\begin{proof}
  Since $A$ is finite, $s_i$ is integral and satisfies a polynomial equation. Since $A[1/s_i]$ is non trivial, we can write this polynomial
  on the form $X^N(1-XP)$. We then have $s_i^N = s_i^N (s_iu)^N$ and so $e_i = e_i^2$ for $e_i = (s_iu)^N$. 
\end{proof}

$A[1/e_i]$ is a extension of the form $k[X]/(p)$. We can always assume that the constant coefficient of $p$ is $\neq 0$. (Since $p$ is
separable, the only problem is in the case $p=X$, but then we replace it by $p = X-1$.)

We can choose $x$ in $B$ such that $x$ generates $A[1/e_i]$. We then have $k[x]e_i = Ae_i$. If we take $y = xe_i$ we still have
$Ae_i = k[y]e_i$ but now, since we have $P(x) = 0$ and the constant coefficient of $P$ is $\neq 0$, it is in $xk[x]$, and so
$e_i$ is in $yk[y]$ and so $Ae_i\subseteq k[y]$.

\section{Lifting to $B$}

$B$ is finite over $R$, and,from the previous section, we have
$y$ in $B$ and $e_i$ in $R[y]$ such that $e_i B\subseteq R[y]+\m B$. By Nakayama's Lemma,
we get $p_i$ in $R[y]$
such that $p_iB\subseteq R[y]$\footnote{Let $b_1,\dots,b_m$
be generators of the $R$-module $B$. We can write $e_i b_j = q_j + \Sigma_l \mu_{jl}b_l$ with
$q_j$ in $R[y]$ and $\mu_{jl}$ in $\m$. We let $p_i$ be the determinant of the matrix $e_i\delta_{jl} - \mu_{jl}$.}.
We then have $B[1/p_i]\cong R[y][1/p_i]$.

To summarize, we have found $t_1,\dots,t_n$ in $B$, comaximal modulo $\m$, such that each $E[1/t_i]$ is of the form
$R[y][1/g]$ for some $y$ in $B$ and $g$ in $R[y]$.

\section{A candidate for a standard \'etale presentation}

 We now consider the \'etale algebra $E[1/t_i]$ and rewrite it as $E = R[y][1/g]$.

 We look at the $R$-algebra $R[y]$. We know that $R[y]/\m R[y]$ is a monogene finite $k$-algebra $k[y]$, but we don't have access
 to its rank constructively in general. We also know that $k[y][1/g]$ is a monogene separable algebra $k[X]/(p)$ for some polynomial $p$.
 This implies that we have $g^np(y) = 0$ in $k[y]$ for some $n$. Let $G$ in $R[X]$ such that $G(y) = g$. 

 Since $R[y]$ is finite, $y$ satisfies a monic equation $Q_0(y) = 0$ for some monic $Q_0$ in $R[X]$. We define recursively the sequence
 of monic polynomials $Q_l$ of decreasing degree $d(Q_l)$ in $R[X]$, all satisfying $Q_l(y) = 0$ in $R[y]$, starting with $Q_0$ and stopping when $d(Q_{l+1})=d(Q_l)$.
  Let $P$ be $gcd(Q_l,G^np)$ in $k[X]$.
  We have $P(y) = 0$ in $k[y]$ since both $Q_l(y) = 0$ and $G(y)^np(y) = 0$. Let $d = d(P)\leqslant d(Q_l)$ be the degree of $P$.
 We have that $1,y,\dots,y^{d-1}$ generate $k[y] = R[y]/\m R[y]$, and so,
 using Nakayama's Lemma, also generate $R[y]$.
 It follows that $y$ also satisfies a monic equation $Q_{l+1}(y) = 0$ for some monic $Q_{l+1}$ in $R[X]$ of degree $d$.
 So after a finite number of steps, we obtain $Q = Q_l$ monic in $R[X]$ such that $Q(y) = 0$ in $k[y]$ and $Q$ divides $G^np$ in $k[X]$.

 Consider now $R[a]= R[X]/(Q)$, and $S= R[a][1/G(a)]$. We have $G^n(a) p(a) = 0$ in $S/\m S$ and $G(a)$ invertible in $S$, so $p(a) = 0$ in $S/\m S$.
 Also, if $q(a) = 0$ in $S/\m S$ for some $q$ in $k[X]$, we have that $Q$ divides $G^mq$ for some $m$ in $k[X]$ and so $q(y) = 0$ in $k[y][1/g]$
 and so $d(p)\leqslant d(q)$.
 It follows that the
 canonical map $u:S\rightarrow R[y][1/g] = E$ is a surjection such that $S/\m S\rightarrow E/\m E$ is an isomorphism\footnote{Note that, contrary
 to Raynaud \cite{Raynaud}, we cannot be sure to have the canonical map $R[a]\rightarrow R[y]$ to be an isomorphism
 residually, since this would amount to be able to compute the rank of $k[y]$; but it becomes a bijection residually when we localise.}.

 All this works so far only assuming $E$ (finitely presented) unramified. 

\section{What happens if $E$ is \'etale}

 We have a surjection $u:S\rightarrow E$, and $S$ and $E$ are finitely presented, so the kernel $I$ is an ideal of $S$
 which is finitely generated. Furthermore $S/\m S = E/\m E = k[y][1/g]$.

 We have the exact sequence $0\rto I/I^2\rto S/I^2\rto E\rto 0$, and, since $E$ is \'etale and $I/I^2$ of square $0$, this exact
 sequence is split (in the sense of existence of a section). Furthermore, if an exact sequence $0\rto A\rto B\rto C\rto 0$ is split, then the sequence
 $0\rto A/\m A\rto B/\m B\rto C/\m C\rto 0$ is still exact.
 Since $S/I^2\rightarrow E$ is an isomorphism mod. $\m$, it follows that $I/I^2$ is $0$ mod. $\m$, and so, since
 $I$ is finitely generated, we have $I = I^2$ by Nakayama. This implies that $I$ is generated by an idempotent element $e = 1-f$.
 Thus we get $f$ in $S$ such that $I[1/f] = 0$ and $u(f) = 1$. Then $S[1/f]$ is isomorphic to $E$ which is standard \'etale.

\section{A global version}

We now take $R$ an arbitrary ring and $E$ a finitely presented \'etale $R$-algebra.

\begin{theorem}
  There exists $s_1,\dots,s_m$ in $E$ that are comaximal and $f_1,\dots,f_m$ in $R$ comaximal in $E$ such that each
  $E[1/s_i]$ is isomorphic to a standard \'etale $R[1/f_i]$-algebra.
\end{theorem}

\begin{proof}
  We introduce a generic (decidable) ideal prime $q$ of $E$, and $p = q\cap R$. We applied the arguments in the previous
  sections to $E_p$ which is \'etale on $R_p$, which is a local ring. We then get $s$ not in $q$ and $f$ in $R$ not in $q$
  such that $E[1/s]$ is isomorphic to a $R[1/f]$ \'etale algebra. Using dynamical methods, we get the statement.
 \end{proof}

It follows, as in Raynaud \cite{Raynaud}, that each \'etale algebra is flat.

We can refine this statement as follows\footnote{I got this argument from Ofer Gabber.}.

\begin{theorem}
  There exists $s_1,\dots,s_m$ in $E$ that are comaximal and $f_1,\dots,f_m$ in $R$ comaximal in $E$ such that each
  $E[1/s_i]$ is isomorphic to a standard \'etale $R$-algebra.
\end{theorem}

\begin{proof}
  Using the previous Theorem, it is enough to show that a standard \'etale $R[1/f]$-algebra is also a standard \'etale $R$-algebra.
  Indeed, for a monic degree $d$ polynomial $P$ over $R[1/f]$,
  one can pass to a monic degree $d$ polynomial over $R$ which lifts $f^{nd} P(X/f^n)$ for $n$ big enough
  and defines the same extension of $R[1/f]$.
\end{proof}

\section{A basic example}

Consider $E = R[X]/(1-Xt)$ with $R,\m$ local and $t$ in $\m$.

If we apply the local version, we get no information since $E/\m E$ is trivial.

If we apply the global version, we get the family $1$ in $E$ and $t$ in $R$ such that $E$ is isomorphic to $R[1/t]$.

\section{What happens if $E$ is flat and unramified}

Following \cite{Raynaud}, we can prove that if $E$ is {\em flat} and (finitely presented) unramified,
then $E$ is locally standard \'etale, and so \'etale.

As before, let $I$ is the kernel of the map $u:S\rightarrow E$. We have the exact sequence
$0\rightarrow I\rightarrow S\rightarrow E\rightarrow 0$.
Furthermore the map $S/\m S\rightarrow E/\m E$ is an isomorphism.
Since $I$ is a finitely generated ideal of $S$, it is enough, by Nakayama, to show that $I/\m I = 0$.
We can then use the following Lemma (instantiated with $A = R$ and $J = \m$) to conclude,
since $E$ is flat over $R$. We use the constructive definition of flatness in \cite{LQ}. 

\begin{lemma}
  Let $L,M,N$ be $A$-module and $J$ an ideal of $A$. We assume $u:L\rightarrow M$ and $v:M\rightarrow N$
  such that $0\rightarrow L\rightarrow M\rightarrow N\rightarrow 0$
  is exact. If the module $N$ is flat, then the map $L/JL\rightarrow M/JM$ is injective.
\end{lemma}

\begin{proof}
   Let $x$ be in $L$ such that $u(x) = 0$ in $M/JM$.
  We then as a relation $u(x) = rm$ for some column vector $r$ in $J$ and some row vector $m$ in $M$.
  We then have $0 = v(u(x)) = rv(m)$ in $N$. Since $N$ is flat, there exists a matrix $p$ in $A$ and
  a row vector $n$ in $N$ such that $rp = 0$ and $pn = v(m)$. Since $v$ is surjective we can find $m_1$ in $M$
  such that $n = v(m_1)$. We then have $v(m-pm_1) = 0$ and so $m-pm_1 = u(l)$ for some $l$ row vector in $L$.
  We then have $u(x) = rm = rm -rpm_1 = ru(l) = u(rl)$ and so $x = rl$ since $u$ is injective. But this means
  that $x=0$ in $L/JL$ as wanted.
\end{proof}

\end{document}